\documentclass[a4paper,11pt]{amsart}
\usepackage{amssymb}
\usepackage[dvips]{graphicx}
\usepackage{amscd}
\usepackage[latin1]{inputenc}
\newcommand{\comm}[1]{}

\def\ti{\tilde}

\def\({\left(}
\def\){\right)}
\def\oli{\overline}

\def\raw{\rightarrow}

\def\no={\neq}
\def\sm{\setminus}

\def\C{{\mathbb C}}

\def\BB{{\mathcal B}}

\def\NN{{\mathcal N}}

\def\RR{{\mathcal R}}

\def\de{\delta}

\def\vep{\varepsilon}

\def\la{\lambda}

\def\om{\omega}

\def\La{\Lambda}

\theoremstyle{plain}
\newtheorem{Main}{Theorem}

\newtheorem{Thm}{Theorem}[section]

\newtheorem{Lem}[Thm]{Lemma}

\theoremstyle{remark}

\newtheorem{Def}[Thm]{Definition}

\begin{document}

\title{Semi-hyperbolic maps are rare}
\author{Magnus Aspenberg}

\begin{abstract}
We prove in this paper that the set of semi-hyperbolic rational maps has Lebesgue measure zero in the space of rational maps of the Riemann sphere of a fixed degree $d \geq 2$, generalising \cite{AG}.
\end{abstract}

\maketitle

\section{Introduction}

In this paper we study critically non-recurrent rational maps on the Riemann sphere. One class of such maps is the semi-hyperbolic maps, which were introduced by Carleson, Jones and Yoccoz in \cite{JJ}. In this paper they proved, among other things, that every Fatou component for a semi-hyperbolic polynomial is a John domain. Let $J(f)$ and $F(f)$ be the Julia set and Fatou set respectively for $f$ and let $Crit(f)$ be the set of critical points of $f$ on the Riemann sphere (we use the spherical metric, so that for instance all polynomials have a critical point at infinity). Let $\om(c)$ be the omega-limit set of $c$, i.e. the set of limit points of the forward orbit of $c$.
\begin{Def}
Suppose $f$ is a non-hyperbolic rational map without parabolic periodic points. Then we say that 
\begin{itemize}
\item $f$ is a {\em semi-hyperbolic map} if $c \notin \om(c)$, for all $c \in Crit(f) \cap J(f)$. 
\item $f$ is a {\em Misiurewicz map} if $\om(c) \cap Crit(f) = \emptyset$ for all $c \in Crit(f) \cap J(f)$.
\end{itemize}
\end{Def}
Hence the Misiurewicz maps have the stronger type of condition that the set of critical points on the Julia set is non-recurrent instead of just every single critical point on the Julia set. This means that the forward orbit of critical points cannot accumulate on other critical points, which could be the case for semi-hyperbolic maps.

Misiurewicz maps are special kind of Collet-Eckmann maps, which are maps for which every critical point $c$ in the Julia set satisfies the so called {\em Collet-Eckmann condition}, namely that there exist $\la > 1$ and $C > 0$ such that $|(f^n)'fc)| \geq C\la^n$ for all $n \geq 0$, whenever there is no critical point in the forward orbit of $c$. 
Semi-hyperbolic maps do not have this property necessarily, although they have some expanding properties. By \cite{GS}, Fatou components of rational Collet-Eckmann maps have some geometric regularity (they are H\"older domains). In \cite{NM} and later \cite{NM-JJ}, N. Mihalache extended \cite{GS} and \cite{JJ} and showed that Fatou components for rational maps satisfying the so called recurrent Collet-Eckmann condition are John domains. (The recurrent CE-condition means that every critical point on the Julia set either satisfies the Collet-Eckmann condition or is non-recurrent). Hence, as a special case, Fatou components for rational semi-hyperbolic maps also possess this geometric regularity.

If we turn to the measure theoretic picture, non-recurrent dynamics seems to be a rare event in the parameter space. 
In the real quadratic family, see \cite{DS}, Sands proved a conjecture by Misiurewicz that the set of Misiurewicz maps has Lebesgue measure zero, and in \cite{SZ}, Zakeri proved that the set of Misiurewicz maps has full Hausdorff dimension, i.e. equal to $1$. For uni-critical complex polynomials, similar results were obtained by J. Rivera-Letelier \cite{RL2}. In \cite{MA3} it was proven that Misiurewicz maps have measure zero in the parameter space of rational maps of any fixed degree $d \geq 2$ and in \cite{AG} that they have full Hausdorff dimension,
i.e. equal to the dimension of the parameters space. For the exponential family $f_{\la}(z) = \la e^z$, $\la \in \C$ similar results also hold, see \cite{Agnieszka} and \cite{Dobbs1-CMP}. We prove in this paper the following.

\begin{Main} \label{semirare}
The set of semi-hyperbolic maps has Lebesgue measure zero in the space $\RR^d$ of rational functions for any fixed degree $d \geq 2$.
\end{Main}

We note here that Theorem \ref{semirare} was shown to be true for $d=2$ already in \cite{AG}, so this paper generalises \cite{AG} to all degrees. 

The space $\RR^d$ of rational maps of degree $d \geq 2$ is identified with the coefficient space, where the coefficients are $a_j, b_j \in \C$ and the rational function is
\[
f(z) = \frac{a_0 + a_1z + \ldots + a_d z^d}{b_0 + b_1z + \ldots + b_d z^d},
\]
and where at least one of $a_d$ and $b_d$ is non-zero. Moreover, we require that there are no common zeros of the denominator and the numerator. This turns the space $\RR^d$ into a ($2d+1$)-complex dimensional manifold with two charts (corresponding to $a_d \neq 0$ and $b_d \neq 0$).

\section{Proof of Theorem \ref{semirare}}

We will first make an equivalent definition of semi-hyperbolic maps following Definition 0.2 in \cite{MA3}. We first let $S^d$ be the set of semi-hyperbolic maps of degree $d$. Let
\[
P^k(f,c) = \oli{\bigcup_{n > k} f^n(c)}.
\]
Let $SupCrit(f)$ the set of critical points of $f$ which lie in a super-attracting cycle.

\begin{Def}
Let $\ti{S}^d$ be the set of non-hyperbolic rational maps of degree $d$ without parabolic cycles and such that $c \notin \om(c)$ for all critical points not laying in super-attracting cycles. We put further
\begin{align}
S_{\de,k} &= \{ f \in \ti{S}^d: P^k(f,c) \cap B(c,\de)  = \emptyset \text{ for all $c \in Crit(f) \sm SupCrit(f)$ }\} \nonumber \\
S_{\de} &= \cup_{k \geq 0} S_{\de,k}. \nonumber
\end{align}
\end{Def}

It is easy to see that $\ti{S}^d = S^d$. We note further that every semi-hyperbolic map $f$ has some $\de > 0$ and $k \geq 0$ such that $f \in S_{\de,k}$.

To avoid some unnecessary work we start by reducing the space of rational maps and only consider maps which are not conjugate to each other by a M\"obious transformation. Hence we consider $\ti{\RR}^d = \RR^d /\sim$ where $\sim$ is the equivalence relation defined by $f \sim g$ if and only if $f = M^{-1}\circ g \circ M$ for some M\"obious transformation $M$. The new space $\ti{\RR}^d$ has dimension $2d-2$, which is equal to the number of critical points, including multiplicity. 

Denote by $c_1, \ldots, c_q$ the critical points on the Julia set for a semi-hyperbolic map $f$. Then they are all non-recurrent. Suppose we have a small analytic family $f_a$, $|a| < \vep$ of maps around $f=f_0$ where $a \in \C^{2d-2}$. By for instance \cite{McM-book} the functions $\xi_{n,j}(a) = f_a^n(c_j(a))$ are all normal families for each fixed $j$ in a small disk around $a=0$ if and only if the Julia set $J(f)$ is $J$-stable, i.e. moves holomorphically.

Since every $c_j$ is non-recurrent the omega limit set $\om(c_j)$ has the property $c_j \notin \om(c_j)$. Moreover, we
have the following evident transitive property, that if the forward orbit of $c_1$ accumulates on $c_2$ then it also accumulates on $\om(c_2)$. Hence if  $c_i \in \om(c_j)$ then we cannot have $c_j \in \om(c_i)$.

Hence the critical points $c_1, \ldots, c_q$ are vertices in a directed graph, where each directed edge $c_ic_j$ from $c_i$ to $c_j$ corresponds to the fact that $c_i$ accumulates on $c_j$. Since $c_i \in \om(c_j)$ implies $c_j \notin \om(c_i)$ there cannot be any loops in this graph. Hence there has to be at least one element, call it $c_p$, which has the property that
\[
c \notin \om(c_p), \qquad \text{for all critical points $c$}.
\]

The closure of the forward orbit of $c_p$,
\[
P(f,c_p) = \oli{\bigcup_{n > 0} f^n(c_p)},
\]
is then a forward invariant, compact subset of $J(f)$ and does not contain any critical or (by definition) parabolic periodic points. By a Theorem of Mane \cite{RM}, $P(f,c_p)$ has to be a uniformly expanding (so called hyperbolic) set. This means that there exists some $N > 0$ and $\la > 1$ such that
\[
|f(^n)'(z)| \geq \la \qquad \text{ for every $z \in P(f,c_p)$}.
\]

It is well known that there is a holomorphic motion of this set. Put $\La = P(f,c_p)$. There is a function $h: \La \times \BB(0,\vep) \raw \hat{\C}$ such that
\[
f_a \circ h_a(z) = h_a \circ f_0(z),
\]
where $h_a(z) = h(z,a)$ and $\BB(0,\vep)$ is a $2d-2$-dimensional ball in the parameter space $\ti{\RR}_d$ centered at $f=f_0$. Moreover, $h$ is analytic in $a$ and quasi-conformal with bounded dilatation in $z$, for fixed $a \in \BB(0,\vep)$ for some $\vep > 0$.
Suppose for simplicity that $c_p$ is simple, so that it moves holomorphically also (and do not split into many critical points). Higher order critical points are degenerate in the parameter space and rational functions having such critical points take up only a set of Lebesgue measure zero, so we may well assume that all critical points are simple.

Moreover, for technical reasons, we may assume that the set $\La$ and does not contain $\infty$ and that $\infty$ is not a critical point. Since $\La \cup Crit(f) \neq \hat{\C}$ this can be done by pre- and post-composing $f$ with a suitable M\"obious transformation. Assuming this is done, we may use the spherical metric and standard metric interchangeably on $\La$ since they are equivalent on $\La$.

Introduce the parameter functions
\[
x_p(a) = f_a(c_p(a)) - h_a(f_0(c_p(0)).
\]
Then $x_p(a)$ is analytic in $\BB(0,\vep)$. We want to prove that $x_p(a)$ is not identically equal to zero.

\begin{Lem}
The function $x_p(a)$ is not identically equal to zero in $\BB(0,\vep)$.
\end{Lem}
\begin{proof}
Suppose the contrary, i.e. that $x_p(a)$ is identically equal to zero. Then clearly, $\xi_{n,p}$ is a normal family in $\BB(0,\vep)$. 

If all $\xi_{n,j}(a)$, $j=1, \ldots, q$ are normal then the $J(f)$ moves holomorphically and we have a contradiction, since semi-hyperbolic maps cannot carry an invariant line field on its Julia set unless $f$ is a Latt\'es map, by Lemma 2.7 in \cite{AG}, and the fact that the Julia set of a semi-hyperbolic map has measure zero if it is not the whole sphere. See also Lemma 2.3 \cite{MA3} (this lemma applies to semi-hyperbolic maps as well as Misiurewicz maps) or \cite{Levin-book}.

Now we use a result by G. Levin \cite{Levin-book} Theorem 1. We formulate a direct consequence of this Theorem here. Suppose that given a critical point $c$ is weakly expanding, namely satisfying the condition,
\[
\sum_{j=0}^{\infty} \frac{1}{|(f^j)'(f(c))|} < \infty.
\]
Then there is a $1$-dimensional complex manifold $M$ going through $f=f_0$, such that when a small neighbourhood of $f$ in $M$ is parameterised non-degenerately, and $f_t \in M$ for any $t$ inside a $1$-dimensional disk $B(0,\vep) \subset \BB(0,\vep)$, then we have that the following limit exists and is non-zero:
\[
\lim\limits_{n \raw \infty} \frac{\xi_{n,p}'(0)}{(f^n)'(fc)} = L \neq 0.
\]
We note here that in \cite{Levin-book} standard derivatives are used. But since we assume that $\La$ is a bounded set in the plane the spherical and standard derivatives are comparable. 

We apply this result to $c=c_p$ and see directly that since $\La$ is uniformly expanding, 
we evidently have $|(f^n)'(fc)| \raw \infty$. Hence also $\xi_{n,p}'(0) \raw 0$.
But this means that $\xi_{n,p}$ cannot be a normal family in $B(0,\vep)$ and hence not in $\BB(0,\vep)$. Hence we have a contradiction and $x_p$ is not identically equal to zero.
\end{proof}

Since $x_p(a)$ is not identically equal to zero we conclude that, since $x_p$ is analytic in several variables, for almost all complex $1$-dimensional disks $B(0,\vep) \subset \BB(0,\vep)$, the restriction of $x_p$ to $B(0,\vep)$ is not identically zero. We now may use the following distortion lemma from \cite{MA3}, Lemma 3.5. Before we formulate it, we say that a disk $D \in B(0,\vep)$ is called a $k$-Whitney disk if
\[
diam(D) \geq k \cdot dist(D,0).
\]
We let $\NN$ be a neighbourhood of $\La$ such that still $|(f^N)'(z)| \geq \la_0 > 1$, for some $\la_0 > 1$ for all $z \in \La$ and for all $a \in \BB(0,\vep)$.
\begin{Lem} \label{distlemma}
If $\vep > 0$ is sufficiently small, there exists a number $0 < k < 1$ only depending on the function $x_p$, and a number $S  > 0$ such that the following holds for any $k$-Whitney disk inside $B(0,\vep)$: There is an $n > 0$ such that the set $\xi_{n,p}(D) \subset \NN$  has diameter at least $S$ and
\[
\biggl| \frac{\xi_{k,p}'(a)}{\xi_{k,p}'(b)} - 1 \biggr| \leq \frac{1}{100}
\]
for all $a,b \in D$ and all $k \leq n$.
\end{Lem}
Basically this lemma says that the distortion of $\xi_{n,p}$ on $k$-Whitney disks is very small up to the large scale (which is denoted by $S$ in the above lemma).

We may now conclude the proof partially inspired by \cite{MA3}. First, by Lemma 4.1, for any $r > 0$ there is some $\ti{N}$ such that any disk $D$ of radius $r > 0$ centered at the Julia set of $f$ has the property that $f^n(D)$ covers the whole Riemann sphere for some $n \leq \ti{N}$. We may choose $\vep$ such that this property persists for small perturbations, in a parameter disk $\BB(0,\vep)$.

Let us say that a semi-hyperbolic is $(\de,k)$-semi-hyperbolic if for every critical point $c \in J(f)$ we have that $B(c,\de) \cap P^k(f,c) = \emptyset$. Hence, $f \in S_{\de,k}$, defined in the beginning of this section.  It is also clear that we only need to prove that the set $S_{\de,k}$ has Lebesgue measure zero for every $\de > 0$ and $k \geq 1$ by taking the union
\[
S^d = \bigcup_{n,k=1}^{\infty} S_{1/n,k}.
\]

We now use induction over the number of critical points in the Julia set $J(f_0)$. To start the induction, we will assume that are $2$ critical points in the Julia set for $f_0$. Suppose that $f=f_0$ is a $\de$-semi-hyperbolic map.

Before we start the induction, we proceed as follows. Assume that the set of $\de$-semi-hyperbolic maps (of a fixed degree $d$) with at most $q-1$ critical points in their Julia set has Lebesgue measure zero.
Suppose that $f=f_0$ is a $\de$-semi-hyperbolic map for some $\de > 0$ and has $q$ critical points on the Julia set as before, with $c_p$ being a critical point whose forward orbit does not contain any other critical point.
Now let $D_0 = B(a_0,r_0)$ be a $k$-Whitney disk inside $B(0,\vep)$ where $x_p$ is not identically zero. By Lemma \ref{distlemma} there is some $n > 0$ such that $\xi_{n,p}(D_0)$ has diameter at least $S/2$. Put $D_1 = B(a_0,r_0/2)$ and $D_2 = B(a_0,r_0/4)$ and let $D' = \xi_{n,p}(D_1)$ and $D'' = \xi_{n,p}(D_2)$.

Let us now note that for any disk $D_z' \subset F(f_0)$ centered at $z$ of diameter $S' > 0$ there is a perturbation $\vep > 0$ such that $D_z'' \subset F(f_a)$ for all $a \in \BB(0,\vep)$ where $D_z''$ is a disk of diameter $S'/2$ centered at $z$. Moreover, by compactness we may  choose $\vep > 0$ such that this holds independently of $z$ as long as $D_z' \subset F(f_0)$.
We now have two cases.

{\bf Case 1.} If $D' \subset F(f_0)$, we have seen from above that there is some $\vep > 0$ such that $D'' \subset F(f_a)$ for all $a \in \BB((0,\vep)$. Hence the set of maps inside $D_2$ have at most $q-1$ critical point on their Julia set. By induction the set of $\de$-semi-hyperbolic maps inside $D_2$ has Lebesgue measure zero. Since $\mu(D_2) = (1/8)\mu(D_0)$ we have
\[
\mu(\{ a \in D_0 : f_a \in S_{\de} \}) \leq \mu(D_0 \sm D_2) = (7/8) \mu(D_0).
\]

{\bf Case 2.} If $D' \cap J(f_0) \neq \emptyset$ then we know that $\xi_n(D_0)$ contains a disk of radius at least $S/4$ centered at the Julia set of $f_0$. Hence $\xi_{n+m}(D_0)$ covers the whole Riemann sphere for some $m \leq \ti{N}$. Consequently $\xi_{n+m}(D_0)$ also covers $B(c_p,\de)$. By bounded distortion, there is a constant $0 < C < 1$ such that
\[
\mu(\xi_{n+m,p}^{-1}(B(c_p,\de)) \geq C \mu(D_0).
\]
For any parameter $a \in \xi_{n+m,p}^{-1}(B(c_p,\de)$, $f_a$ cannot be $(\de,n+m-1)$-semi-hyperbolic unless $\xi_{n+m,p}(a) \in B(c_p(0),\de/2)$ and $c_p(a)$ has become super-attracting. But parameters for which there is a super-attracting periodic orbit have Lebesgue measure zero, so we have
\[
\mu(\{ a \in D_0 : f_a \in S_{\de,n+m-1} \}) \leq (1-C)\mu(D_0).
\]

Hence in all cases we have that for each $k$-Whitney disk in $B(0,\vep)$ the Lebesgue measure of the set of $(\de,n+m-1)$-semi-hyperbolic maps is strictly less than $C \mu(D_0)$, where $0 < C < 1$.

We now start the induction and proceed in the same way. If there are only $2$ critical points $c_1$ and $c_2$ in the Julia set for $f=f_0$ then we assume that the forward orbit of $c_2$ does not accumulate on $c_1$ (and $c_2$ by definition) so that $\La = P(f,c_2)$ is a hyperbolic set as before. With the same notations as above, let $D_0 = B(a_0,r_0)$ be a $k$-Whitney disk inside $B(0,\vep)$ where $x_2$ is not identically zero. 

In Case 1 above we have that there is a set $D_2 \subset D_0$ where each $a \in D_2$ corresponds to a parameter for which there is at most one critical point in the Julia set. In this case semi-hyperbolic maps and Misiurewicz maps coincide. We know from \cite{MA3} that such maps has Lebesgue measure zero and thus
\[
\mu(\{ a \in D_0 : f_a \in S_{\de} \}) \leq  (7/8) \mu(D_0).
\]
In Case 2 above we may use exactly the same argument (in fact the induction is not used in this case) to conclude that 
\[
\mu(\{ a \in D_0 : f_a \in S_{\de,n+m-1} \}) \leq (1-C)\mu(D_0).
\]

These estimates hold for every $k$-Whitney disk $D_0 \subset B(0,\vep)$ around the starting map $f=f_0$, for which the number of critical points on its Julia set for the starting map $f=f_0$ can be at most $q$, given that the set of $\de$-semi-hyperbolic maps has Lebesgue measure zero if there are at most $q-1$ critical points on the Julia set.

We finish the proof by first noting that the sets $S_{\de,k}$ are nested so that $S_{\de,k} \subset S_{\de,k+1}$. Moreover, the number $k=m+n-1$ can be made arbitrarily large if $\vep$ is sufficiently small. Altogether, assuming that the set of $\de$-semi-hyperbolic maps with maximum $q-1$ critical points on the Julia set has Lebesgue measure zero, this means that the Lebesgue density of $(\de,k)$-semi-hyperbolic maps in $B(0,\vep)$ with maximum $q$ critical points on the Julia set has to be strictly smaller than $1$ for any $k \geq 1$ and $\de > 0$. Hence the set of $(\de,k)$-semi-hyperbolic maps (with maximum $q$ critical points on the Julia set) has Lebesgue measure zero for any $k \geq 1$ and $\de > 0$ and consequently the same holds for the set of $\de$-semi-hyperbolic maps (with maximum $q$ critical points on its Julia set). By induction over $q$ this completes the proof of Theorem \ref{semirare}.

\bibliographystyle{plain}
\bibliography{ref}

\end{document}